\documentclass[11pt]{amsart}
\usepackage[margin=1in]{geometry}
\usepackage{amssymb,amsfonts,amsthm,amsrefs,mathtools}
\usepackage{color}

\DeclareFontFamily{U}{mathb}{\hyphenchar\font45}
\DeclareFontShape{U}{mathb}{m}{n}{
      <5> <6> <7> <8> <9> <10> gen * mathb
      <10.95> mathb10 <12> <14.4> <17.28> <20.74> <24.88> mathb12
      }{}
\DeclareSymbolFont{mathb}{U}{mathb}{m}{n}
\DeclareFontSubstitution{U}{mathb}{m}{n}
\DeclareMathSymbol{\monus}{2}{mathb}{"01}

\newtheorem{Theorem}{Theorem}[section]
 
\newtheorem{Lemma}[Theorem]{Lemma}
\newtheorem{Corollary}[Theorem]{Corollary}
\newtheorem{ClaimH}{Claim}
\newtheorem{ClaimT}{Claim}
\newtheorem*{GrowthTheorem}{Growth theorem}
\newtheorem*{GeneralizedGrowthTheorem}{Generalized growth theorem}

\theoremstyle{definition}
\newtheorem{Definition}[Theorem]{Definition}

\newtheorem{Question}[Theorem]{Question}

\DeclareMathOperator{\tot}{\mathrm{tot}}
\DeclareMathOperator{\true}{\mathrm{true}}

\newcommand{\andd}{\wedge}
\newcommand{\orr}{\vee}
\newcommand{\la}{\langle}
\newcommand{\ra}{\rangle}
\newcommand{\lc}{\ulcorner}
\newcommand{\rc}{\urcorner}
\newcommand{\da}{{\downarrow}}
\newcommand{\ua}{{\uparrow}}
\newcommand{\imp}{\rightarrow}

\newcommand{\biimp}{\leftrightarrow}

\newcommand{\join}{\cup}
\newcommand{\meet}{\cap}

\newcommand{\pa}{\mathrm{PA}}
\newcommand{\bso}{\mathrm{B}\Sigma_1}
\newcommand{\iso}{\mathrm{I}\Sigma_1}

\newcommand{\leqE}{\leq_\mathrm{E}}

\newcommand{\geqE}{\geq_\mathrm{E}}
\newcommand{\ngeqE}{\ngeq_\mathrm{E}}
\newcommand{\equivE}{\equiv_\mathrm{E}}
\newcommand{\leE}{<_\mathrm{E}}
\newcommand{\geE}{>_\mathrm{E}}
\newcommand{\llE}{\ll_\mathrm{E}}
\newcommand{\ggE}{\gg_\mathrm{E}}
\DeclareMathOperator{\degree}{\mathrm{deg}}
\DeclareMathOperator{\degE}{\degree_\mathrm{E}}
\newcommand{\dg}[1]{\mathbf{#1}}

\newcommand{\leqaE}[1]{\leq_{#1\mathrm{E}}}

\newcommand{\geqaE}[1]{\geq_{#1\mathrm{E}}}
\newcommand{\ngeqaE}[1]{\ngeq_{#1\mathrm{E}}}
\newcommand{\equivaE}[1]{\equiv_{#1\mathrm{E}}}

\newcommand{\geaE}[1]{>_{#1\mathrm{E}}}
\newcommand{\llaE}[1]{\ll_{#1\mathrm{E}}}
\newcommand{\ggaE}[1]{\gg_{#1\mathrm{E}}}
\newcommand{\degaE}[1]{\degree_{#1\mathrm{E}}}

\newcommand{\leqP}[1]{\leq_{#1}}

\newcommand{\geqP}[1]{\geq_{#1}}

\newcommand{\equivP}[1]{\equiv_{#1}}
\newcommand{\leP}[1]{<_{#1}}
\newcommand{\geP}[1]{>_{#1}}
\newcommand{\degP}[1]{\degree_{#1}}

\newcommand{\run}{\mathtt{runA}}
\newcommand{\btrue}{\mathtt{true}}
\newcommand{\bfalse}{\mathtt{false}}

\newcommand{\slim}{\mathrm{SLim}}

\newcommand{\Hd}{\mathcal{H}}
\newcommand{\Had}[1]{\mathcal{H}_{#1}}
\newcommand{\Pd}[1]{\mathcal{P}_{#1}}

\usepackage{xcolor}	
\usepackage{soul}

\title[Degrees without the cupping property]{Honest elementary degrees and degrees of relative provability without the cupping property}

\author{Paul Shafer}
\address{Department of Mathematics\\
Ghent University\\
Krijgslaan 281 S22\\
B-9000 Ghent\\
Belgium}
\email{paul.shafer@ugent.be}
\urladdr{http://cage.ugent.be/~pshafer/}
\thanks{Paul Shafer is an FWO Pegasus Long Postdoctoral Fellow.}

\date{\today}

\begin{document}

\begin{abstract}
An element $a$ of a lattice \emph{cups} to an element $b > a$ if there is a $c < b$ such that $a \join c = b$.  An element of a lattice has the \emph{cupping property} if it cups to every element above it.  We prove that there are non-zero honest elementary degrees that do not have the cupping property, which answers a question of Kristiansen, Schlage-Puchta, and Weiermann~\cite{KristiansenSchlage-PuchtaWeiermann}.  In fact, we show that if $\dg b$ is a sufficiently large honest elementary degree, then there is an $\dg a$ with $\dg 0 \leE \dg a \leE \dg b$ that does not cup to $\dg b$.  For comparison, we modify a result of Cai~\cite{CaiHigherUnprovability} to show that, in several versions of the related degrees of relative provability, the preceding property holds for all non-zero $\dg b$, not just sufficiently large $\dg b$.
\end{abstract}

\maketitle

\section{Introduction}

An element $a$ of a lattice \emph{cups} to an element $b > a$ if there is a $c < b$ such that $a \join c = b$.  An element $a$ of a lattice has the \emph{cupping property} if it cups to every $b > a$.  In this work, we study the cupping property in several related lattices arising from elementary functions and total algorithms.

The first lattice we consider is the the lattice $\Hd$ of \emph{honest elementary degrees}, which arose from attempts to classify various sub-recursive classes of functions into hierarchies.  In $\Hd$, the objects are (equivalence classes of) functions whose graphs are elementary relations, and these functions are compared via the `elementary in' relation.  The basic theory of this structure was developed by Meyer and Ritchie~\cite{MeyerRitchie} and by Machtey~\cites{MachteyLoop, MachteyLattice, MachteyDensity}.  In recent years, intense work mainly by Kristiansen~\cites{KLSWstructure, KristiansenAlgorithmsVsFunctions, KristiansenDegreesFragments, KristiansenInfo, KristiansenJump, KristiansenLoHi, KristiansenSchlage-PuchtaWeiermann} has significantly advanced the theory.  We refer the reader to~\cite{KLSWstructure} (and to the related~\cite{KristiansenSchlage-PuchtaWeiermann}) for a survey of the area.  In~\cite{KristiansenSchlage-PuchtaWeiermann}, the authors ask if every $\dg a \in \Hd$ with $\dg a \geE \dg 0$ has the cupping property.  We answer this question negatively by showing that if $\dg b \in \Hd$ is sufficiently large (in the sense of Definition~\ref{def-dominate}), then there is an $\dg a \in \Hd$ with $\dg 0 \leE \dg a \leE \dg b$ that does not cup to $\dg b$ (Corollary~\ref{cor-NoCuppingH}).

Next we consider two related families of lattices:  the \emph{degrees of provability} relative to arithmetical theories extending $\iso$ and the \emph{honest $\alpha$-elementary degrees} for ordinals $\alpha \leq \epsilon_0$ of the form $\omega^\beta$.  Let $T$ be a consistent first-order theory in the language of arithmetic.  In $\Pd{T}$, the degrees of provability relative to $T$, the objects are (equivalence classes of) total algorithms (i.e., indices of total Turing machines), and these algorithms are compared via the `provably total' relation.  That is, $\degP{T}(\Phi) \geqP{T} \degP{T}(\Psi)$ if $T \vdash \tot(\Phi) \imp \tot(\Psi)$, where $\tot(\Phi)$ is the sentence expressing the totality of the Turing machine $\Phi$.  Cai~\cite{CaiDegrees} introduced the degrees of relative provability in order to analyze the provability strengths of true $\Pi_2$ sentences or, equivalently, sentences expressing the totality of total algorithms.  This line of research continues impressively in~\cites{ACDLM, CaiProvingUnprovability, CaiHigherUnprovability}.

In $\Hd_\alpha$, the honest $\alpha$-elementary degrees, the objects are again (equivalence classes of) functions whose graphs are elementary, and these functions are compared via the `$\alpha$-elementary in' relation, which coarsens the `elementary in' relation by allowing functions to be iterated $\beta < \alpha$ many times.  Kristiansen, Schlage-Puchta, and Weiermann~\cite{KristiansenSchlage-PuchtaWeiermann} introduced the honest $\alpha$-elementary degrees in order to connect sub-recursive hierarchies to provability in Peano arithmetic ($\pa$).

The degrees of relative provability and the honest $\alpha$-elementary degrees are very closely related.  For a theory $T$, let $T^+$ be the extension of $T$ by all true $\Pi_1$ sentences.  Kristiansen~\cite{KristiansenAlgorithmsVsFunctions} proves that $\Pd{\pa^+}$ and $\Hd_{\epsilon_0}$ are isomorphic, and analogous results should hold for various fragments of $\pa$ and the appropriate ordinals.

Cai~\cite{CaiHigherUnprovability} proves that there are non-zero elements of $\Pd{\pa^+}$ that do not have the cupping property.  It follows from Kristiansen's isomorphism that there are also non-zero elements of $\Hd_{\epsilon_0}$ that do not have the cupping property.  We modify Cai's result to prove that if $T$ is a consistent, recursively axiomatizable theory extending $\iso$, then for every non-zero $\dg b \in \Pd{T^+}$ there is an $\dg a \in \Pd{T^+}$ with $\dg 0 \leP{T^+} \dg a \leP{T^+} \dg b$ that does not cup to $\dg b$ (Corollary~\ref{cor-NoCuppingT}).  Consider then the following two statements:
\begin{itemize}
\item[($\star$)] For every $b$ that is sufficiently large (where the definition `sufficiently large' depends on the lattice in question) there is an $a$ with $0 < a < b$ that does not cup to $b$.
\item[($\dagger$)] For every $b > 0$ there is an $a$ with $0 < a < b$ that does not cup to $b$.
\end{itemize}
Corollary~\ref{cor-NoCuppingH} states that ($\star$) holds in $\Hd$.  By modifying the argument, we also see that ($\star$) holds in the $\Hd_\alpha$'s.  Corollary~\ref{cor-NoCuppingT} states that ($\dagger$) holds in $\Pd{T^+}$ for every consistent, recursively axiomatizable theory $T$ extending $\iso$.  In particular, ($\star$) holds in $\Pd{\pa^+}$ and so, by Kristiansen's isomorphism, also in $\Hd_{\epsilon_0}$.  Thus the natural question is whether or not ($\dagger$) holds in $\Hd$ and in every $\Hd_\alpha$.  We expect that ($\dagger$) holds in many of the $\Hd_\alpha$'s by extending Kristiansen's isomorphism result to fragments of $\pa$.

\section{Honest elementary degrees}

In this section, we provide a basic introduction to the theory of the honest elementary degrees.  Again, we refer the reader to \cites{KLSWstructure, KristiansenSchlage-PuchtaWeiermann} for more comprehensive surveys.

\begin{Definition}\label{def-elementary}{\ }
\begin{itemize}
\item The \emph{elementary functions} are those functions $f \colon \omega^n \imp \omega$ that can be generated from the \emph{initial elementary functions} by the \emph{elementary definition schemes}.

\item The \emph{initial elementary functions} are
\begin{itemize}
\item the projection functions $\ell^k_i$ for all $k > 0$ and $i < k$, where $\ell^k_i(x_0, \dots, x_i, \dots, x_{k-1}) = x_i$;
\item the $0$-ary constants $0$ and $1$; addition ($+$); and truncated subtraction (i.e., monus $\monus$).
\end{itemize}

\item The \emph{elementary definition schemes} are
\begin{itemize}
\item composition:  $f(\vec x) = h(g_0(\vec x), g_1(\vec x), \dots, g_{m-1}(\vec x))$;
\item bounded sum:  $f(\vec x, y) = \sum_{i < y}g(\vec x, i)$; and
\item bounded product:  $f(\vec x, y) = \prod_{i < y}g(\vec x, i)$.
\end{itemize}

\item A relation is \emph{elementary} if its characteristic function is elementary.  

\item A function $f$ has \emph{elementary graph} if the relation $R(\vec x, y) \coloneqq (f(\vec x) = y)$ is elementary.

\item A function $f \colon \omega^n \imp \omega$ is \emph{elementary in} a function $g \colon \omega^k \imp \omega$ ($f \leqE g$) if $f$ can be generated from $g$ and the initial elementary functions by the elementary definition schemes.

\item Functions $f$ and $g$ are \emph{equivalent} ($f \equivE g$) if $f \leqE g$ and $g \leqE f$.
\end{itemize}
\end{Definition}

The elementary functions have nice closure properties, such as closure under bounded search and closure under bounded primitive recursion.  These closure properties lead to useful alternative characterizations.  To wit, the elementary functions are exactly the functions in $\mathcal{E}^3$, which denotes level $3$ of the Grzegorczyk hierarchy.  That is, the elementary functions are the closure of $0$, the successor function, the projection functions, the exponential function $2^x$, and the $\max$ function under composition and bounded primitive recursion.  One can also take advantage of the fact that Kleene's $\mathcal{T}$ predicate is elementary to show that the elementary functions are exactly those functions that can be computed by Turing machines that run in elementary time.  That is, $f$ is elementary if and only if there is a Turing machine computing $f$ that runs in time $O(2_k^n)$ for some $k$, where $2_k$ is the $k$\textsuperscript{th} iterate of the exponential function (so $2_2^n = 2^{2^n}$, $2_3^n = 2^{2^{2^n}}$, and so forth).  See~\cite[Chapter~1]{Rose} for a presentation of the above-mentioned facts.

We study the class of all functions with elementary graphs, quasi-ordered by $\leqE$.  By the discussion in~\cite[Section~1]{KristiansenDegreesFragments}, it suffices to consider the so-called \emph{honest} functions, as for every function $f$ with elementary graph, there is an honest function $g$ with $g \equivE f$.

\begin{Definition}
A function $f \colon \omega^n \imp \omega$ is \emph{honest} if
\begin{itemize}
\item $f$ is unary:  $n=1$;

\item $f$ dominates $2^x$:  $\forall x (f(x) \geq 2^x)$;

\item $f$ is monotone:  $\forall x(f(x) \leq f(x+1))$; and

\item $f$ has elementary graph.
\end{itemize}
\end{Definition}

The idea behind the terminology is that the output of an honest function gives some indication of how long the computation took.  If $f$ is honest, then there is a Turing machine computing $f$ whose runtime is elementary in $f$.  What would be considered \emph{dishonest} is a Turing machine that makes long computations to produce short outputs (see, for example, \cite{AdamsHH}).

We can now define the honest elementary degrees.

\begin{Definition}\label{def-HEDeg}{\ }
\begin{itemize}
\item The \emph{honest elementary degree} of an honest function $f$ is
\begin{align*}
\degE(f) = \{g : \text{$g$ is honest and $g \equivE f$}\}
\end{align*}

\item The set of \emph{honest elementary degrees} is $\Hd = \{\degE(f) : \text{$f$ is honest}\}$.
\end{itemize}
\end{Definition}

The $\leqE$ relation induces a partial order on $\Hd$ in the usual way:  for honest functions $f$ and $g$, define $\degE(f) \leqE \degE(g)$ if $f \leqE g$.  The resulting structure is a distributive lattice with join defined by $\degE(f) \join \degE(g) = \degE(\max[f,g])$ and meet defined by $\degE(f) \meet \degE(g) = \degE(\min[f,g])$, and this lattice has a minimum element $\dg 0 = \degE(2^x)$ (see~\cite{KLSWstructure}).  Here $\max[f,g]$ is the function defined by $\max[f,g](x) = \max(f(x), g(x))$, and the function $\min[f,g]$ is defined analogously.

For a function $f \colon \omega \imp \omega$ and a $k \in \omega$, let $f^k$ denote the $k$\textsuperscript{th} iterate of $f$, defined by $f^0(x) = x$ and $f^{k+1}(x) = f(f^k(x))$.  For functions $f, g \colon \omega \imp \omega$, write $f \leq g$ if $g$ dominates $f$:  $\forall x(f(x) \leq g(x))$.  Kristiansen's growth theorem (\cite{KristiansenInfo}; see \cite[Theorem~2.3]{KLSWstructure}) characterizing the $\leqE$ relation on honest functions in terms of rates of growth is the key tool for working with the honest elementary degrees.

\begin{GrowthTheorem}[\cite{KristiansenInfo}]
If $f$ and $g$ are honest functions, then $f \leqE g$ if and only if $f \leq g^k$ for some $k \in \omega$.
\end{GrowthTheorem}

\section{Honest elementary degrees without the cupping property}

Kristiansen's result \cite[Theorem~3.4]{KristiansenDegreesFragments} (see also \cite[Theorem~10]{KristiansenSchlage-PuchtaWeiermann} and \cite[Theorem~5.3]{KLSWstructure}) states that every honest elementary degree that is sufficiently large has the cupping property, where `sufficiently large' is made precise by the following definition.

\begin{Definition}\label{def-dominate}{\ }
\begin{itemize}
\item For functions $f,g \colon \omega \imp \omega$, define $f \llE g$ if some fixed iterate of $g$ eventually dominates every iterate of $f$:  $(\exists k)(\forall m)(\forall^\infty x)(f^m(x) \leq g^k(x))$.

\item For honest elementary degrees $\degE(f)$ and $\degE(g)$, define $\degE(f) \llE \deg(g)$ if $f \llE g$.
\end{itemize}
\end{Definition}

As an honest function is equivalent to its finite iterations, it is easy to see that $\dg a \llE \dg b$ if and only if there is a $g \in \dg b$ that eventually dominates every $f \in \dg a$.  We refer the reader to~\cites{KristiansenSchlage-PuchtaWeiermann, KLSWstructure} for more information concerning the $\llE$ relation, including its original definition in terms of universal functions.  We remark that although $\leE$ is a dense partial ordering of $\Hd$ by work of Machtey~\cite{MachteyDensity} (see also~\cites{KLSWstructure, KristiansenDegreesFragments}), it is not known whether $\llE$ is a dense partial ordering of $\Hd$ (see~\cite[Section~4]{MeyerRitchie}).  The precise statement of Kristiansen's theorem on cupping is the following.

\begin{Theorem}[{\cite[Theorem~3.4]{KristiansenDegreesFragments}}]\label{thm-KristiansenCupping}
If $\dg a$ and $\dg b$ are honest elementary degrees with $\dg 0 \llE \dg a \leE \dg b$, then $\dg a$ cups to $\dg b$. 
\end{Theorem}

Thus if $\dg a$ is an honest elementary degree with $\dg a \ggE \dg 0$, then $\dg a$ has the cupping property.  On the other hand, $\dg 0$, being the minimum degree, certainly does not have the cupping property.  Kristiansen, Schlage-Puchta, and Weiermann~\cite{KristiansenSchlage-PuchtaWeiermann} (and again Kristiansen, Lubarsky, Schlage-Puchta, and Weiermann~\cite{KLSWstructure}) therefore ask if Theorem~\ref{thm-KristiansenCupping} can be improved to all $\dg a \geE \dg 0$.  We prove that this is not the case.

Our technical theorem says that if $\dg b \ggE \dg 0$, then there is a $\dg a \geE \dg 0$ that can only cup to degrees $\geqE \dg b$ via degrees that are already $\geqE \dg b$.  Once we have this theorem, it is easy to produce a non-zero $\dg a \leE \dg b$ that does not cup to $\dg b$ by appealing to the distributive lattice structure of $\Hd$.

Let $g \in \dg b$.  To prove the theorem, we need to produce an honest $f \geE 2^x$ such that for every honest $h$, if $\max[f, h] \geqE g$ then $h \geqE f$.  Over the course of its computation, $f$ keeps track of a set $C$ of (indices of) functions $h$ that look like they might satisfy $\max[f, h]^e \geq g$ for some $e$.  Here $\max[f, h]^e$ is the $e$\textsuperscript{th} iterate of the function $\max[f, h]$.  For each $h \in C$, $f$ tries to stay below $h$ so that if $\max[f,h]^e$ really is $\geq g$, then $h$ will eventually dominate $f$.  By removing $h$ from $C$ when learning that $\max[f,h]^e \ngeq g$, $f$ can find safe numbers $x$ for which $f(x)$ can be large in order to ensure that $f \geE 2^x$.

\begin{Theorem}\label{thm-NoCupHelpH}
For every $\dg b \in \Hd$ with $\dg b \ggE \dg 0$, there is an $\dg a \in \Hd$ with $\dg a \geE \dg 0$ such that $(\forall \dg c \in \Hd)[(\dg a \join \dg c \geqE \dg b) \imp (\dg c \geqE \dg b)]$.
\end{Theorem}

\begin{proof}
For notational ease, we intentionally conflate a total Turing machine with the function that it computes.  Let $(\Phi_e : e \in \omega)$ be the usual effective list of all Turing machines.  For each $\Phi_e$, let $\widehat{\Phi}_e$ be the Turing machine that, on input $n$, runs $\Phi_e$ on inputs $0, 1, \dots, n$ and, if all of these computations halt, outputs the maximum of $2^n$ and the total number of steps that the $\Phi_e$ computations took.  The Turing machine $\widehat{\Phi}_e$ is essentially the \emph{honest associate} of $\Phi_e$ as defined in~\cite{KristiansenAlgorithmsVsFunctions}, and we have that
\begin{itemize}
\item if $\widehat{\Phi}_e$ is total, then it is honest; and
\item if $h$ is honest, then there is an $e$ such that $h \equivE \widehat{\Phi}_e$ (see~\cite[Lemma~4]{KristiansenAlgorithmsVsFunctions}).
\end{itemize}
We also think of $\widehat{\Phi}_e(n)$ as being the number of steps in the computation of $\widehat{\Phi}_e(n)$ because the runtime of $\widehat{\Phi}_e$ is $O(\widehat{\Phi}_e)$ (and the constants are independent of $e$).

Let $\Gamma$ be a Turing machine computing a representative of $\dg b$ that, by the assumption $\dg b \ggE \dg 0$, eventually dominates every elementary function.  Define a Turing machine $\Psi$ that behaves as follows on input $n$.

\begin{itemize}
\item Initialize $k \coloneqq 2$, $M \coloneqq 1$, and $C \coloneqq \{0\}$.

\item Main loop:  for each $m \leq n$, run $\widehat{\Phi}_e(m)$ for all $e \in C$ in a dovetailing fashion for at most $2_k^m$ steps each.
\begin{itemize}
	\item If some $\widehat{\Phi}_e(m)$ halts with output $N$:
		\begin{itemize}
			\item For each $e \in C$ and each $\ell < m$:
			\begin{itemize}
				\item Run $\Gamma(\ell)$ and $\max[\Psi, \widehat{\Phi}_e]^e(\ell)$ for $m$ steps each, aborting the computation of $\max[\Psi, \widehat{\Phi}_e]^e(\ell)$ if it produces numbers $\geq m$.  (Observe that the value of $\Psi(\ell)$ is the value of $M$ after iteration $\ell$ of the main loop, so the computation of $\max[\Psi, \widehat{\Phi}_e]^e(\ell)$ can be facilitated by storing the previous values of $M$ in a table.)
				\item If $\Gamma(\ell)$ and $\max[\Psi, \widehat{\Phi}_e]^e(\ell)$ both halt within $m$ steps and $\max[\Psi, \widehat{\Phi}_e]^e(\ell) < \Gamma(\ell)$, then set $C \coloneqq C \setminus \{e\}$.
			\end{itemize}
			\item Set $M \coloneqq \max\{M, N, 2^m\}$.
			\item Stop running the $\widehat{\Phi}_e(m)$'s, and go to the next iteration of the main loop.
		\end{itemize}
	\item Else:
		\begin{itemize}
			\item Set $k \coloneqq k+1$.
			\item Let $i$ be the least number that has never been in $C$, and 				set $C \coloneqq C \cup \{i\}$.
			\item Set $M \coloneqq \max\{M, 2_k^m\}$.
		\end{itemize}
\end{itemize}
\item Output $M$ when the main loop terminates.
\end{itemize}

\begin{ClaimH}\label{claim-honest}
$\Psi$ is honest.
\end{ClaimH}

\begin{proof}
Clearly $\Psi$ is unary.  For a given input $n$, let $M_m$ denote the value of $M$ after iteration $m \leq n$ of the main loop.  It is easy to see that $M_m \geq 2^m$, that $M_m$ is monotonic in $m$, and that $M_m = \Psi(m)$.  Thus $\Psi$ dominates $2^x$ and is monotonic.  We need to show that $\Psi$ has elementary graph.  Recall from the discussion following Definition~\ref{def-elementary} that the elementary functions are exactly the functions that can be computed in elementary time.  Thus we need to show that the graph of $\Psi$ is computable in elementary time.  As the output of $\Psi$ is always bigger than the corresponding input, it suffices to show that the runtime of $\Psi(n)$ is elementary in the value of $\Psi(n)$.  In fact, we show that the runtime of $\Psi$ is polynomial in its outputs.

Let $C_m$ be the value of $C$ at the beginning of iteration $m$ of the main loop.  Notice that at most one number is added to $C$ during each iteration, so $|C_m| \leq m+1$.  In iteration $m$, either there is an $e \in C_m$ such that $\widehat{\Phi}_e(m)$ halts within $2_k^m$ steps, or there is not.  Consider first the case in which there is an $e_0 \in C_m$ is such that $\widehat{\Phi}_{e_0}(m)$ halts within $2_k^m$ steps with output $N$.  Then $\widehat{\Phi}_{e_0}(m)$ halts within $O(N)$ steps, so each $\widehat{\Phi}_e(m)$ with $e \in C_m$ is run for $O(N)$ steps.  Therefore $O(|C_m|N)$ steps are spent running the $\widehat{\Phi}_e(m)$'s.  Afterward, for each $e \in C$ and each $\ell < m$, $\Gamma(\ell)$ and $\max[\Psi, \widehat{\Phi}_e]^e(\ell)$ are run for at most $m$ steps each and compared.  This takes $O(|C_m|m^2)$ steps.  Thus the total number of steps taken in this case is
\begin{align*}
O(|C_m|N + |C_m|m^2) = O(mN + m^3) = O(M_n^3) = O(\Psi(n)^3),
\end{align*}
where the first equality is because $|C_m| \leq m+1$, and the second equality is because $m \leq 2^m \leq M_m \leq M_n$ and, in this case, $N \leq M_m \leq M_n$.

Now consider the case in which no $\widehat{\Phi}_e(m)$ halts within $2_k^m$ steps.  In this case, $O(|C_m|2_k^m)$ steps are spent running the $\widehat{\Phi}_e(m)$'s.  Thus the total number of steps taken is $O(M_n^2) = O(\Psi(n)^2)$ because, in this case, $|C_m| \leq m+1 \leq 2_k^m \leq M_m \leq M_n$.

Thus iteration $m$ of the main loop takes $O(\Psi(n)^3)$ steps.  So $\Psi$ runs in time $O(\Psi(n)^4)$.
\end{proof}

\begin{ClaimH}\label{claim-NotZeroH}
$\Psi \geE 2^x$.
\end{ClaimH}

\begin{proof}
It suffices to show that $k$ increases infinitely often in the sense that for every $n_0$ there is an $n \geq n_0$ such that, in the execution of $\Psi(n)$, the value of $k$ increases at the end of iteration $n$ of the main loop.  This is because $\Psi(n) \geq 2_k^n$ in this case, and, therefore, if $k$ increases infinitely often, then $\forall k \exists n (\Psi(n) \geq 2_k^n)$.  This implies that $\Psi \geE 2^x$ by the growth theorem.

To show that $k$ increases infinitely often, we show that for every $n_0$ there is an $n \geq n_0$ such either $k$ increases or $|C|$ decreases during iteration $n$ of the main loop.  Suppose for a contradiction that there is an $n_0$ such that $k$ never increases and $|C|$ never decreases after iteration $n_0$.  Then iteration $n$ enters the `if' case of the main loop for all $n \geq n_0$.  This implies that there is a fixed $k_0$ such that, for all $n \geq n_0$, $\min\{\widehat{\Phi}_e(n) : e \in C\} < 2_{k_0}^n$ and $\Psi(n) = \max\{\Psi(n-1), N, 2^n\}$ for some $N < 2_{k_0}^n$.  It follows that $\Psi$ is elementary by the growth theorem.  Furthermore, $C$ never changes after iteration $n_0$ because no numbers are removed from $C$ by assumption, and no numbers are added to $C$ because the main loop never enters the `else' case.  Thus there must be an $e$ in this fixed $C$ such that $\widehat{\Phi}_e \not\ggE 0$ because there must be an $e \in C$ for which $\widehat{\Phi}_e(n) < 2_{k_0}^n$ holds for infinitely many $n$.  For this $e$, $\max[\Psi, \widehat{\Phi}_e]^e \equivE \widehat{\Phi}_e \not\ggE 0$, so, because $\Gamma \ggE 0$, there is an $\ell$ such that $\max[\Psi, \widehat{\Phi}_e]^e(\ell) < \Gamma(\ell)$.  Therefore $e$ is removed from $C$ during iteration $n$ of the main loop once $n$ is large enough so that the computation witnessing that $\max[\Psi, \widehat{\Phi}_e]^e(\ell) < \Gamma(\ell)$ takes at most $n$ steps.  This contradicts that $|C|$ never decreases after iteration $n_0$.

Now it is easy to see that $k$ increases infinitely often.  If not, there is an $n_0$ such that $k$ never increases after iteration $n_0$.  In this case, by the preceding argument, $|C|$ can only decrease, so there is an $n \geq n_0$ such that $C$ is empty at the start of iteration $n$.  In this situation, the main loop enters the `else' case, and $k$ is increased, contradicting that $k$ never increases.
\end{proof}

\begin{ClaimH}\label{claim-avoidH}
For every honest $h$, either $\max[\Psi, h] \ngeqE \Gamma$ or $h \geqE \Psi$.
\end{ClaimH}

\begin{proof}
First, consider an index $e$ of a total $\widehat{\Phi}_e$.  From the proof of the previous claim, $k$ increases infinitely often, which implies that $e$ is eventually added to $C$.  If $e$ is never removed from $C$, then $\Psi$ is $O(\max[\widehat{\Phi}_e, 2^x])$, which implies that $\widehat{\Phi}_e \geqE \Psi$.  On the other hand, if $e$ is eventually removed from $C$, then there is an $\ell$ such that $\max[\Psi, \widehat{\Phi}_e]^e(\ell) < \Gamma(\ell)$.

Now consider an honest $h$ and the infinitely many indices $e$ such that $\widehat{\Phi}_e \equivE h$ and $h \leq \widehat{\Phi}_e$.  If one such $e$ enters $C$ and is never removed, then $h \equivE \widehat{\Phi}_e \geqE \Psi$.  If every such $e$ is eventually removed from $C$ after it enters, then for infinitely many $e$ there is an $\ell$ such that $\max[\Psi, h]^e(\ell) \leq \max[\Psi, \widehat{\Phi}_e]^e(\ell) < \Gamma(\ell)$.  Hence $\max[\Psi, h] \ngeqE \Gamma$ by the growth theorem.
\end{proof}

Let $\dg a = \degE(\Psi)$.  Then $\dg a \in \Hd$ by Claim~\ref{claim-honest}, and $\dg a \geE \dg 0$ by Claim~\ref{claim-NotZeroH}.  If $\dg c \in \Hd$ is such that $\dg a \join \dg c \geqE \dg b$, then $\dg c \geqE \dg a$ by Claim~\ref{claim-avoidH}.  Therefore $\dg c \geqE \dg a \join \dg c \geqE \dg b$ as desired, which completes the proof.
\end{proof}

Theorem~\ref{thm-NoCupHelpH} implies that there is an honest elementary degree $\dg a \geE \dg 0$ that does not have the cupping property.  In fact, for every $\dg b \ggE \dg 0$, there is a non-zero $\dg a \leE \dg b$ that does not have the cupping property as witnessed by $\dg b$.

\begin{Corollary}\label{cor-NoCuppingH}
For every $\dg b \in \Hd$ with $\dg b \ggE \dg 0$, there is an $\dg a \in \Hd$ with $\dg 0 \leE \dg a \leE \dg b$ such that $(\forall \dg c \in \Hd)[(\dg a \join \dg c = \dg b) \imp (\dg c = \dg b)]$.  That is, for every $\dg b \ggE \dg 0$, there is a non-zero $\dg a \leE \dg b$ that does not cup to $\dg b$.
\end{Corollary}

\begin{proof}
Given $\dg b \ggE \dg 0$, by Theorem~\ref{thm-NoCupHelpH}, let $\dg x \geE \dg 0$ be such that $(\forall \dg c \in \Hd)[(\dg x \join \dg c \geqE \dg b) \imp (\dg c \geqE \dg b)]$.  Let $\dg a = \dg x \meet \dg b$.  One readily checks that $\dg b \ggE \dg 0$ and $\dg x \geE \dg 0$ imply that $\dg a \geE \dg 0$.  Consider a $\dg c \in \Hd$ such that $\dg a \join \dg c = \dg b$.  Clearly $\dg c \leqE \dg b$.  On the other hand, using the fact that $\Hd$ is a distributive lattice,
\begin{align*}
\dg b = \dg a \join \dg c = (\dg x \meet \dg b) \join \dg c = (\dg x \join \dg c) \meet (\dg b \join \dg c) = (\dg x \join \dg c) \meet \dg b.
\end{align*}
Thus $\dg x \join \dg c \geqE \dg b$, which implies that $\dg c \geqE \dg b$ by the choice of $\dg x$.  Thus $\dg c = \dg b$.
\end{proof}

One could also prove Corollary~\ref{cor-NoCuppingH} by directly ensuring that $\dg a \leqE \dg b$ in the proof of Theorem~\ref{thm-NoCupHelpH} (by enforcing that $\Gamma$ dominates $\Psi$).

\begin{Question}
In Theorem~\ref{thm-NoCupHelpH} and Corollary~\ref{cor-NoCuppingH}, can $\dg b \ggE \dg 0$ be weakened to $\dg b \geE \dg 0$?
\end{Question}

We have shown that there are non-zero honest elementary degrees that do not have the cupping property.  For the curious readier, we briefly summarize what is known about \emph{capping} in $\Hd$.  An element $a$ of a lattice \emph{caps} to $b < a$ if there is a $c > b$ such that $a \meet c = b$.  If $\dg a, \dg b \in \Hd$ are such that $\dg b \llE \dg a$, then $\dg a$ does not cap to $\dg b$~\cite{KristiansenLoHi}.  However, this result does not provide a characterization of capping in $\Hd$ because there are $\dg a, \dg b \in \Hd$ with $\dg b \leE \dg a$ and $\dg b \not\llE \dg a$ such that $\dg a$ does not cap to $\dg b$~\cite{KLSWstructure}.

\section{Degrees of relative provability and honest $\alpha$-elementary degrees}

In this section, we provide a basic introduction to the theory of the degrees of relative provability and the honest $\alpha$-elementary degrees.  We assume familiarity with Peano arithmetic and its fragments.  The most important fragment for us is $\iso$, which consists of the basic axioms and the induction scheme for $\Sigma_1$ formulas.  $\iso$ is \emph{$\Sigma_1$-complete}, meaning that $\iso$ proves every true $\Sigma_1$ sentence.  $\iso$ also suffices to define the $\Pi_1$ truth predicate `$\true(n)$,' which states that $n$ codes (i.e., is the G\"odel number of) a true $\Pi_1$ sentence (see, for example,~\cite[Section~I.1(d)]{HajekPudlak} for details).  We note that the $\Pi_1$ truth predicate is itself $\Pi_1$.  Throughout this section, every theory is assumed to be in the language of arithmetic, to be consistent, and to extend $\iso$.

First we describe Cai's degrees of relative provability~\cite{CaiDegrees}.

\begin{Definition}
Fix a theory $T$.  For a Turing machine $\Phi$, let $\tot(\Phi)$ be the $\Pi_2$ sentence expressing that $\Phi$ is total.
\begin{itemize}
\item Turing machine $\Phi$ \emph{provably reduces} to Turing machine $\Psi$ ($\Phi \leqP{T} \Psi$) if $T \vdash \tot(\Psi) \imp \tot(\Phi)$.

\item Turing machines $\Phi$ and $\Psi$ are \emph{provably equivalent} ($\Phi \equivP{T} \Psi$) if $\Phi \leqP{T} \Psi$ and $\Psi \leqP{T} \Phi$.

\item The \emph{provability degree} of a Turing machine $\Phi$ is
\begin{align*}
\degP{T}(\Phi) = \{\Psi : \text{$\Psi$ is a Turing machine and $\Psi \equivP{T} \Phi$}\}
\end{align*}

\item The set of \emph{provability degrees} is $\Pd{T} = \{\degP{T}(\Phi) : \text{$\Phi$ is a total Turing machine}\}$.
\end{itemize}
\end{Definition}

We follow the usual convention that a Turing machine halts on input $n$ if and only if it halts on all inputs $m \leq n$.  This is without loss of generality, assuming $\iso$.  For a Turing machine $\Phi$, let $\widehat\Phi$ be the Turing machine that, on input $n$, runs $\Phi(0), \Phi(1), \dots, \Phi(n)$ in succession and halts if and only if $(\forall m \leq n)(\Phi(m)\da)$.  Then $T \vdash \tot(\widehat\Phi) \biimp \tot(\Phi)$.

It is easy to see that $\leqP{T}$ quasi-orders the Turing machines and therefore induces a partial order on $\Pd{T}$.  In fact, $\Pd{T}$ is a distributive lattice.  Let $\Phi$ and $\Psi$ be two total Turing machines.  Then $\degP{T}(\Phi) \join \degP{T}(\Psi) = \degP{T}(\Gamma)$, where, for each $n$, $\Gamma(n)$ runs $\Phi(n)$ and $\Psi(n)$ simultaneously and converges when both $\Phi(n)$ and $\Psi(n)$ converge.  Similarly, $\degP{T}(\Phi) \meet \degP{T}(\Psi) = \degP{T}(\Theta)$, where, for each $n$, $\Theta(n)$ runs $\Phi(n)$ and $\Psi(n)$ simultaneously and converges when either $\Phi(n)$ converges or $\Psi(n)$ converges.  Notice that $T \vdash \tot(\Gamma) \biimp (\tot(\Phi) \andd \tot(\Psi))$ and that $T \vdash \tot(\Theta) \biimp (\tot(\Phi) \orr \tot(\Psi))$.  $\Pd{T}$ also has a minimum element $\dg 0$, which is the degree of any Turing machine that $T$ proves is total, such as the machine that immediately halts and outputs $0$ on every input.  See~\cite{CaiDegrees} for proofs of these facts.  

We remark that if $\varphi$ is a true $\Pi_2$ sentence, then there is a total Turing machine $\Phi$ such that $T \vdash \varphi \biimp \tot(\Phi)$.  Thus one may think of $\Pd{T}$ as the Lindenbaum algebra of $T$ restricted to true $\Pi_2$ sentences.

Now we describe Kristiansen, Schlage-Puchta, and Weiermann's honest $\alpha$-elementary degrees~\cite{KristiansenSchlage-PuchtaWeiermann}.  First, we recall that every ordinal $\alpha < \epsilon_0$ has a Cantor normal form $\omega^{\alpha_0} + \omega^{\alpha_1} + \dots + \omega^{\alpha_{n-1}} + 0$, where the ordinals $\alpha_0 \geq \alpha_1 \geq \dots \geq \alpha_{n-1}$ are themselves in Cantor normal form.  This allows us to define the \emph{norm} of an ordinal $\alpha < \epsilon_0$ by induction on its Cantor normal form.

\begin{Definition}
Let $\alpha < \epsilon_0$.  The \emph{norm} of $\alpha$, $N(\alpha)$, is defined by induction on $\alpha$'s Cantor normal form by $N(0) = 0$, $N(\beta + \gamma) = N(\beta) + N(\gamma)$, and $N(\omega^\beta) = 1 + N(\beta)$.
\end{Definition}

This definition of norm allows us to make sense of iterating a function transfinitely many times.

\begin{Definition}
Let $f \colon \omega \imp \omega$, and let $\alpha < \epsilon_0$.  The $\alpha$\textsuperscript{th} iterate of $f$, $f_\alpha$, is defined inductively by
\begin{align*}
f_0(n) &= f(n) &\\
f_\alpha(n) &= \max\{f_\beta(f_\beta(n)) : (\beta < \alpha) \andd (N(\beta) \leq N(\alpha) + n)\} & \text{for $\alpha > 0$}.
\end{align*}
\end{Definition}

We now define the `$\alpha$-elementary in' relation by adding transfinite iteration to the elementary definition schemes of Definition~\ref{def-elementary}.

\begin{Definition}
Let $\alpha \leq \epsilon_0$.
\begin{itemize}
\item A function $f \colon \omega^n \imp \omega$ is \emph{$\alpha$-elementary in} a function $g \colon \omega^k \imp \omega$ ($f \leqaE{\alpha} g$) if $f$ can be generated from $g$ and the initial elementary functions of Definition~\ref{def-elementary} by the elementary definition schemes of Definition~\ref{def-elementary} and by $\beta$-iteration for all $\beta < \alpha$.

\item Functions $f$ and $g$ are \emph{equivalent} ($f \equivaE{\alpha} g$) if $f \leqaE{\alpha} g$ and $g \leqaE{\alpha} f$.
\end{itemize}
\end{Definition}

We note that if $f$ is honest and $\alpha < \epsilon_0$, then $f_\alpha$ is also honest~\cite{KristiansenSchlage-PuchtaWeiermann}.

Let $\slim = \{\alpha \leq \epsilon_0 : (\exists \beta > 0)(\alpha = \omega^\beta)\}$.  For $\alpha \in \slim$, Kristiansen, Schlage-Puchta, and Weiermann give the following generalization of the growth theorem.

\begin{GeneralizedGrowthTheorem}[\cite{KristiansenSchlage-PuchtaWeiermann}]
If $\alpha \in \slim$ and $f$ and $g$ are honest functions, then $f \leqaE{\alpha} g$ if and only if $f \leq g_\beta$ for some $\beta < \alpha$.
\end{GeneralizedGrowthTheorem}

We define the honest $\alpha$-elementary degrees for $\alpha \in \slim$ analogously to Definition~\ref{def-HEDeg}.

\begin{Definition}
Let $\alpha \in \slim$.
\begin{itemize}
\item The \emph{honest $\alpha$-elementary degree} of an honest function $f$ is
\begin{align*}
\degaE{\alpha}(f) = \{g : \text{$g$ is honest and $g \equivaE{\alpha} f$}\}
\end{align*}

\item The set of \emph{honest $\alpha$-elementary degrees} is $\Had{\alpha} = \{\degaE{\alpha}(f) : \text{$f$ is honest}\}$.
\end{itemize}
\end{Definition}

Again, $\Had{\alpha}$ is a distributive lattice with partial order induced by $\leqaE{\alpha}$, join and meet defined via $\max$ and $\min$ as with $\Hd$, and minimum element $\dg 0 = \degaE{\alpha}(2^x)$~\cite{KristiansenSchlage-PuchtaWeiermann}.  Notice that $\Had{\omega} = \Hd$ because the finite iterates of a function can be defined using the elementary definition schemes.

For a theory $T$, let $T^+$ denote $T$ extended by all true $\Pi_1$ sentences.  The connection between the degrees of relative provability and the honest $\alpha$-elementary degrees is made clear by the following result of Kristiansen.

\begin{Theorem}[\cite{KristiansenAlgorithmsVsFunctions}]\label{thm-iso}
$\Pd{\pa^+}$ and $\Had{\epsilon_0}$ are isomorphic.
\end{Theorem}

$\Pd{T^+}$ and $\Had{\alpha}$ should also be isomorphic for various fragments $T$ of $\pa$ and the appropriate ordinals $\alpha$, but the details still need to be checked.  However, we do not know which, if any, of the $\Pd{T^+}$'s are isomorphic and which, if any, of the $\Had{\alpha}$'s are isomorphic.

\begin{Question}
Are the $\Pd{T^+}$'s isomorphic for the fragments $T$ of $\pa$ extending $\iso$?  Are they elementarily equivalent?  Are the $\Had{\alpha}$'s isomorphic for the $\alpha \in \slim$?  Are they elementarily equivalent?
\end{Question}

In the following section, we extend a result of Cai's implying that there are non-zero elements of $\Pd{T^+}$ that do not have the cupping property~\cite{CaiHigherUnprovability}.  Thus, by Theorem~\ref{thm-iso}, there are non-zero elements of $\Had{\epsilon_0}$ that do not have the cupping property.  We can also prove this fact directly by running the proof of Theorem~\ref{thm-NoCupHelpH} in the context of $\Had{\epsilon_0}$.  In fact, the proof of Theorem~\ref{thm-NoCupHelpH} can be modified to show that, for every $\alpha \in \slim$, there are non-zero elements of $\Had{\alpha}$ that do not have the cupping property.

Fix $\alpha \in \slim$, and reinterpret `$\ll$' in the context of $\Had{\alpha}$ by defining $f \llaE{\alpha} g$ to mean that there is a $\gamma < \alpha$ such that $g_\gamma$ eventually dominates $f_\beta$ for every $\beta < \alpha$:  $(\exists \gamma < \alpha)(\forall \beta < \alpha)(\forall^\infty x)(f_\beta(x) \leq g_\gamma(x))$.  In particular, $\dg b \ggaE{\alpha} \dg 0$ means that there is a $g \in \dg b$ that eventually dominates $2_\beta^x$ (the $\beta$\textsuperscript{th} iterate of $2^x$) for every $\beta < \alpha$.

As in the proof of Theorem~\ref{thm-NoCupHelpH}, fix a Turing machine $\Gamma$ computing a member of $\dg b$ that eventually dominates $2_\beta^x$ for every $\beta < \alpha$.  Fix an elementary \emph{fundamental sequence} of ordinals $\alpha_0 < \alpha_1 < \dots$ that converges to $\alpha$.  That is, fix an elementary function that maps $k$ to a code for $\alpha_k$.  Define the Turing machine $\Psi$ as before, except now replace $2_k^m$ by $2_{\alpha_k}^m$ and replace $\max[\Psi, \widehat{\Phi}_e]^e$ by $\max[\Psi, \widehat{\Phi}_e]_{\alpha_e}$.  $\Psi$ is again honest because its runtime is elementary in its outputs.  In order to honestly compute the $2_{\alpha_k}^m$'s, use the fact that if $f$ is honest, then the predicate ``$f_{\alpha_k}(x) = y$'' is elementary, which is proven in the course of the proof of~\cite[Lemma~17]{KristiansenSchlage-PuchtaWeiermann}.  Again, $\Psi \geaE{\alpha} 2^x$ because $k$ must increase infinitely often.  That $k$ increases infinitely often implies that $\Psi$ is not dominated by $2_{\alpha_k}^x$ for any $k$, which, as $\lim_k \alpha_k = \alpha$, implies that $\Psi$ is not dominated by $2_\beta^x$ for any $\beta < \alpha$.  Finally, for any honest $h$, either $\max[\Psi, h] \ngeqaE{\alpha} \Gamma$ or $h \geqaE{\alpha} \Psi$ by again considering the indices $e$ for which $\widehat{\Phi}_e \equivaE{\alpha} h$ and whether or not they are all eventually removed from $C$.

\section{Degrees of relative provability without the cupping property}

Throughout this section, we assume that all theories considered are consistent theories in the language of arithmetic.  In~\cite{CaiHigherUnprovability}, Cai proves that if $T$ is a recursively axiomatizable extension of $\pa$, then there are non-zero elements of $\Pd{T^+}$ that do not have the cupping property.  We modify Cai's proof in a few ways that we hope will be helpful in future work comparing the degrees of relative provability to the honest $\alpha$-elementary degrees.

First, we prove a stronger statement:  for every non-zero $\dg b \in \Pd{T^+}$, there is a non-zero $\dg a \leP{T^+} \dg b$ that does not cup to $\dg b$.  Second, we work with $\Pd{T^+}$ directly.  Cai considers many different subalgebras of $T$'s Lindenbaum algebra, and he proves that, in the subalgebra of true $\Pi_1$ sentences, there are non-zero elements that do not have the cupping property.  He then obtains the corresponding result for $\Pd{T^+}$ by relativization and an application of an isomorphism theorem.  We use Cai's same strategy to directly define a total Turing machine $\Psi$ whose degree does not non-trivially cup above a given non-zero degree $\dg b$.  This construction is slightly more complicated than Cai's original construction, but we believe it has some benefits in addition to being technically interesting in its own right.  The direct construction is easier to see in terms of Kristiansen's isomorphism from Theorem~\ref{thm-iso}, and we hope that it will help decide whether or not $\dg b \ggaE{\alpha} \dg 0$ can be replaced by $\dg b \geaE{\alpha} \dg 0$ in the $\Had{\alpha}$ cases.  The direct construction also makes it a little easier to keep track of how much of $T$ is being used.  In our proof, we only assume that $T$ extends $\iso$, whereas Cai assumes that $T$ extends $\pa$ (though for Cai this assumption is mostly a matter of convenience).

We say that a number $t$ \emph{witnesses} a $\Sigma_1$ sentence $\exists n \varphi(n)$ if $(\exists n < t)\varphi(n)$.  That is, $t$ witnesses $\exists n \varphi(n)$ if $t$ is large enough verify that $\exists n \varphi(n)$ is true.  Let $\lc \varphi \rc$ denote the code of the formula $\varphi$ according to some fixed G\"odel numbering.  We first define a helpful family of auxiliary Turing machines.  For a $\Pi_1$ sentence $\eta$ and a finite set $C$ of pairs of the form $\la \lc \pi \rc, e \ra$, where $\pi$ is a $\Pi_1$ sentence and $e$ is the index of a Turing machine, let $A^{\lc \eta \rc}_C$ be the Turing machine that behaves as follows on input $s$.

\begin{itemize}
\item Initialize $t \coloneqq s$.

\item While $\btrue$:
\begin{itemize}
\item If $t$ witnesses $\neg\eta$:  halt and output $0$.
\item If there is a $\la \lc \pi \rc, e \ra \in C$ such that $(\forall n \leq s)(\Phi_{e,t}(n)\da)$ and $t$ does not witness $\neg\pi$:  halt and output $0$.
\item Else:  set $t \coloneqq t+1$.
\end{itemize}
\end{itemize}
\noindent
(The notation `$\Phi_{e,t}(n)\da$' means that the execution of $\Phi_e(n)$ halts within $t$ steps.)

Many of the following arguments combine reasoning in ordinary mathematics with reasoning inside of a formal theory.  We warn the reader that, to keep notational clutter to a minimum, we intentionally conflate the number $s \in \omega$ with the standard term that names it.  For example, if we have determined that the Turing machine $\Phi$ halts on all inputs $n \leq s$ and then want to reason formally about this, we write `$(\forall n \leq s)(\Phi(n)\da)$' instead of the more technically correct `$(\forall n \leq \overline s)(\Phi(n)\da)$,' where $\overline s$ is the name for $s$.

\begin{Lemma}\label{lem-helperA}
Let $T$ be a recursively axiomatizable extension of $\iso$.
\begin{itemize}
\item[(1)]
\begin{align*}
T^+ \vdash (\forall C)(\forall \la \lc \pi \rc, e \ra \in C)[(\true(\lc \pi \rc) \andd \tot(\Phi_e)) \imp \tot(A^{\lc \eta \rc}_C)].
\end{align*}

\item[(2)] Let $\eta$ be a true $\Pi_1$ sentence, and let $C$ be a finite set of pairs.  Then
\begin{align*}
T^+ \vdash \tot(A^{\lc \eta \rc}_C) \biimp \bigvee_{\substack{\la \lc \pi \rc, e \ra \in C\\ \textup{$\pi$ is true}}}\tot(\Phi_e)
\end{align*}
\textup{(}where the empty disjunction is considered to be false\textup{)}.
\end{itemize}
\end{Lemma}

\begin{proof}
First we prove (1).  Working in $T^+$, suppose that $\la \lc \pi \rc, e \ra \in C$, $\true(\lc \pi \rc)$, and $\tot(\Phi_e)$.  Consider an arbitrary $s$.  By $\tot(\Phi_e)$ and $\bso$, there is a $t \geq s$ such that $(\forall n \leq s)(\Phi_{e,t}(n)\da)$.  Such a $t$ does not witness $\neg\pi$ by the assumption $\true(\lc \pi \rc)$.  Therefore there is a least $t \geq s$ such that either $t$ witnesses $\neg\eta$ or there is a $\la \lc \pi' \rc, e' \ra \in C$ such that $(\forall n \leq s)(\Phi_{e',t}(n)\da)$ and $t$ does not witness $\neg\pi'$.  Therefore $A^{\lc \eta \rc}_C(s)\da$.  As $s$ is arbitrary, we conclude $\tot(A^{\lc \eta \rc}_C)$.

The preceding argument also proves the `$\leftarrow$' direction of (2) because if $\pi$ is true, then $\true(\lc \pi \rc)$ is a true $\Pi_1$ sentence and hence an axiom of $T^+$.  So we need to prove the `$\imp$' direction of (2).  We prove the contrapositive.  Work in $T^+$ and suppose that
\begin{align*}
\bigwedge_{\substack{\la \lc \pi \rc, e \ra \in C\\ \textup{$\pi$ is true}}}\neg\tot(\Phi_e).
\end{align*}
For the conjunct indexed by $\la \lc \pi \rc, e \ra$, let $s_e$ be such that $\forall t(\Phi_{e,t}(s_e)\ua)$.  Let $s$ be larger than all of the $s_e$'s and large enough to witness $\neg\pi$ for all $\la \lc \pi \rc, e \ra \in C$ with $\pi$ false.  Note that no $t$ can witness $\neg\eta$ because $\eta$ is assumed to be true, and thus $\true(\lc \eta \rc)$ is an axiom of $T^+$.  Thus $A^{\lc \eta \rc}_C(s)\da$ leads to a contradiction, so we must have that $A^{\lc \eta \rc}_C(s)\ua$.  Therefore $\neg\tot(A^{\lc \eta \rc}_C)$.
\end{proof}

The next theorem is our modification of Cai's~\cite[Theorem~7.1]{CaiHigherUnprovability} and an analog of Theorem~\ref{thm-NoCupHelpH}.

\begin{Theorem}\label{thm-NoCupHelpT}
Let $T$ be a recursively axiomatizable extension of $\iso$.  For every $\dg b \in \Pd{T^+}$ with $\dg b \geP{T^+} \dg 0$, there is an $\dg a \in \Pd{T^+}$ such that $\dg a \meet \dg b \geP{T^+} \dg 0$ and $(\forall \dg c \in \Pd{T^+})[(\dg a \join \dg c \geqP{T^+} \dg b) \imp (\dg c \geqP{T^+} \dg b)]$.
\end{Theorem}

\begin{proof}
Let $\Gamma$ be a total Turing machine with $\degP{T^+}(\Gamma) = \dg b \geP{T^+} \dg 0$.  Define a Turing machine $\Psi$ that behaves as follows on input $s$.  The definition of $\Psi$ uses the recursion theorem (see~\cite[Theorem~II.3.1]{Soare}) to assume that $\Psi$ has access to its own code.
\begin{itemize}
\item Initialize $\run \coloneqq \bfalse$, $C \coloneqq \emptyset$, $p \coloneqq 0$, $\lc \eta \rc \coloneqq \lc 0 = 0 \rc$.
\item Main loop:  for each $m \leq s$, do the following:
	\begin{itemize}
		\item If $\run$ is $\btrue$ and $m$ does not witness $\neg\eta$:  			execute $A^{\lc \eta \rc}_C(m)$.
		\item Else:
			\begin{itemize}
				\item Set $\run \coloneqq \bfalse$.
				\item If $p$ codes a proof witnessing
				\begin{align*}
				T + \pi \vdash (\tot(\Phi_e) \andd \tot(\Psi)) \imp \tot(\Gamma)
				\end{align*}
				for some $\Pi_1$ sentence $\pi$ and some $e$:  set 					$C \coloneqq C \cup \{\la \lc \pi \rc, e \ra\}$.
				\item If $p$ codes a proof witnessing
				\begin{align*}
				T + \eta' \vdash \tot(\Psi) \orr \tot(\Gamma)
				\end{align*}
				for some $\Pi_1$ sentence $\eta'$:  set $\lc \eta \rc 					\coloneqq \lc \eta' \rc$, and set $\run \coloneqq \btrue$.
				\item Set $p \coloneqq p+1$.
			\end{itemize}
	\end{itemize}
\item Output $0$.
\end{itemize}

\begin{ClaimT}\label{claim-tot}
$\Psi$ is total.
\end{ClaimT}

\begin{proof}[Proof of claim]
Suppose for a contradiction that $\Psi(s)\ua$ for some $s$.  Observe that $\Psi(s)\ua$ is a true $\Pi_1$ sentence and hence an axiom of $T^+$.  Thus $T^+ \vdash \Psi(s)\ua$, so $T^+ \vdash \neg\tot(\Psi)$ (in fact, $T^+$ proves all true $\Sigma_2$ sentences by the same argument).  As $\Psi(s)\ua$, it must be that the execution of $\Psi(s)$ executes $A^{\lc \eta \rc}_C(m)$ for some $\eta$, $C$, and $m$ for which $A^{\lc \eta \rc}_C(m)\ua$.  For this to happen, it must be that $T + \eta \vdash \tot(\Psi) \orr \tot(\Gamma)$.  Furthermore, it is easy to see that $A^{\lc \eta \rc}_C$ is total if $\eta$ is false.  So it must be that $\eta$ is true, in which case $T^+ \vdash \tot(\Psi) \orr \tot(\Gamma)$.  From $T^+ \vdash \neg\tot(\Psi)$ and $T^+ \vdash \tot(\Psi) \orr \tot(\Gamma)$ we conclude that $T^+ \vdash \tot(\Gamma)$, which is a contradiction.
\end{proof}

\begin{ClaimT}\label{claim-NotZeroT}
$T^+ \nvdash \tot(\Psi) \orr \tot(\Gamma)$.
\end{ClaimT}

\begin{proof}[Proof of claim]
Suppose for a contradiction that $T^+ \vdash \tot(\Psi) \orr \tot(\Gamma)$, and let $p_0$ be the least number coding a proof witnessing that $T + \eta_0 \vdash \tot(\Psi) \orr \tot(\Gamma)$ for some true $\Pi_1$ sentence $\eta_0$.  Then if $p < p_0$ codes a proof witnessing that $T + \eta \vdash \tot(\Psi) \orr \tot(\Gamma)$ for some $\Pi_1$ sentence $\eta$, this $\eta$ must be false.  Therefore, if the main loop is iterated enough times, $p$ is eventually set to $p_0$, and $\lc \eta \rc$ is eventually set to $\lc \eta_0 \rc$.  Let $s_0$ be least such that $p$ is set to $p_0$ and $\lc \eta \rc$ is set to $\lc \eta_0 \rc$ during the execution of $\Psi(s_0)$, and let $C_0$ be the value of $C$ when $p$ is set to $p_0$.  That $s_0$ has its defining property is (equivalent to) a true $\Sigma_1$ sentence, so $T^+$ proves that $s_0$ is least such that $p$ is set to $p_0$ and $\lc \eta \rc$ is set to $\lc \eta_0 \rc$ during the execution of $\Psi(s_0)$.

Now we work in $T^+$ to show that $T^+ \vdash (\forall s > s_0)(\Psi(s)\da) \biimp (\forall s > s_0)(A^{\lc \eta_0 \rc}_{C_0}(s)\da)$.  First suppose that $(\forall s > s_0)(\Psi(s)\da)$.  Consider the execution of $\Psi(s)$ for an $s > s_0$.  We know that $p$ is set to $p_0$ and that $\lc \eta \rc$ is set to $\lc \eta_0 \rc$ during iteration $s_0$ of the main loop.  Also, no number witnesses $\neg\eta_0$ because $\eta_0$ is true and hence $T^+ \vdash \true(\lc \eta_0 \rc)$.  Therefore, the main loop enters the `if' case in all iterations past $s_0$.  In particular, the main loop executes $A^{\lc \eta_0 \rc}_{C_0}(s)$ in iteration $s$.  Thus $A^{\lc \eta_0 \rc}_{C_0}(s)\da$ because $\Psi(s)\da$.  Conversely, suppose that $(\forall s > s_0)(A^{\lc \eta_0 \rc}_{C_0}(s)\da)$.  By the preceding claim, $(\forall s \leq s_0)(\Psi(s)\da)$, which is (equivalent to) a true $\Sigma_1$ sentence.  Thus $T^+ \vdash (\forall s \leq s_0)(\Psi(s)\da)$.  We prove by $\Sigma_1$ induction on $s$ that $(\forall s \geq s_0)(\Psi(s)\da)$.  We already know that $\Psi(s_0)\da$, which gives the base case.  Now assume that $\Psi(s)\da$, and consider the execution of $\Psi(s+1)$.  The execution of $\Psi(s+1)$ reaches iteration $s+1$ of the main loop because $\Psi(s)\da$.  As argued above, the fact that $s+1 > s_0$ means that the main loop executes $A^{\lc \eta_0 \rc}_{C_0}(s+1)$ in iteration $s+1$.  By assumption $A^{\lc \eta_0 \rc}_{C_0}(s+1)\da$, so $\Psi(s+1)\da$.

Now, from
\begin{align*}
T^+ &\vdash (\forall s > s_0)(\Psi(s)\da) \biimp (\forall s > s_0)(A^{\lc \eta_0 \rc}_{C_0}(s)\da),\\
T^+ &\vdash (\forall s \leq s_0)(\Psi(s)\da), \text{ and}\\
T^+ &\vdash (\forall s > s_0)(A^{\lc \eta_0 \rc}_{C_0}(s)\da) \biimp \tot(A^{\lc \eta_0 \rc}_{C_0})
\end{align*}
(the last of which is easy to see), we conclude that $T^+ \vdash \tot(\Psi) \biimp \tot(A^{\lc \eta_0 \rc}_{C_0})$.  Therefore, by Lemma~\ref{lem-helperA}~item~(2),
\begin{align*}
T^+ \vdash \tot(\Psi) \biimp \bigvee_{\substack{\la \lc \pi \rc, e \ra \in C_0\\ \textup{$\pi$ is true}}}\tot(\Phi_e).
\end{align*}
If the disjunction is empty, then $T^+ \vdash \neg\tot(\Psi)$.  Combining this with the assumption $T^+ \vdash \tot(\Psi) \orr \tot(\Gamma)$ yields $T^+ \vdash \tot(\Gamma)$, which is a contradiction.  If the disjunction is not empty, then consider each $\la \lc \pi \rc, e \ra \in C_0$ where $\pi$ is true.  For $\la \lc \pi \rc, e \ra$ to have been added to $C_0$, it must be that $T + \pi \vdash (\tot(\Phi_e) \andd \tot(\Psi)) \imp \tot(\Gamma)$.  Therefore $T^+ \vdash (\tot(\Phi_e) \andd \tot(\Psi)) \imp \tot(\Gamma)$ because $\pi$ is true.  Thus
\begin{align*}
T^+ &\vdash \tot(\Psi) \imp \bigvee_{\substack{\la \lc \pi \rc, e \ra \in C_0\\ \textup{$\pi$ is true}}}(\tot(\Phi_e) \andd \tot(\Psi)), \text{ and}\\
T^+ &\vdash \left(\bigvee_{\substack{\la \lc \pi \rc, e \ra \in C_0\\ \textup{$\pi$ is true}}}(\tot(\Phi_e) \andd \tot(\Psi))\right) \imp \tot(\Gamma).
\end{align*}
It follows that $T^+ \vdash \tot(\Psi) \imp \tot(\Gamma)$.  Combining this with the assumption $T^+ \vdash \tot(\Psi) \orr \tot(\Gamma)$ yields $T^+ \vdash \tot(\Gamma)$, which is a contradiction.  Thus $T^+ \nvdash \tot(\Psi) \orr \tot(\Gamma)$, as desired.
\end{proof}

\begin{ClaimT}\label{claim-avoidT}
If $e$ is such that $T^+ \vdash (\tot(\Phi_e) \andd \tot(\Psi)) \imp \tot(\Gamma)$, then $T^+ \vdash \tot(\Phi_e) \imp \tot(\Psi)$.
\end{ClaimT}

\begin{proof}[Proof of claim]
We start by showing that $p$ increases infinitely often in the sense that for every $p_0$ there is an $m$ such that $p$ is set to $p_0$ in iteration $m$ of the main loop.  First, $\Psi$ is total by Claim~\ref{claim-tot}, so $\Psi$ never diverges during the execution of the main loop.  Second, $p$ increases exactly in iterations where the main loop enters the `else' case.  Thus if $p$ increases only finitely often, there must be an $m_0$ such that main loop only enters the `if' case in iterations past $m_0$.  For this to happen, there must be a true $\Pi_1$ sentence $\eta'$ such that $T + \eta' \vdash \tot(\Psi) \orr \tot(\Gamma)$.  Thus $T^+ \vdash \tot(\Psi) \orr \tot(\Gamma)$, which contradicts Claim~\ref{claim-NotZeroT}.

Now, suppose that $T^+ \vdash (\tot(\Phi_e) \andd \tot(\Psi)) \imp \tot(\Gamma)$, and let $p_0$ be a proof witnessing that $T + \pi \vdash (\tot(\Phi_e) + \tot(\Psi)) \imp \tot(\Gamma)$ for some true $\Pi_1$ sentence $\pi$.  Let $s_0$ be such that $p$ is increased from $p_0$ to $p_0+1$ during iteration $s_0$ of the main loop, so that $\la \lc \pi \rc, e \ra$ is added to $C$ during this iteration.  We now argue in $T^+$ that $\tot(\Phi_e) \imp \tot(\Psi)$.  As argued in Claim~\ref{claim-NotZeroT}, $(\forall s \leq s_0)(\Psi(s)\da)$ is (equivalent to) a true $\Sigma_1$ sentence, so $T^+ \vdash (\forall s \leq s_0)(\Psi(s)\da)$.  We prove by $\Sigma_1$ induction on $s$ that $(\forall s \geq s_0)(\Psi(s)\da)$.  We already know that $\Psi(s_0)\da$, which gives the base case.  Now assume that $\Psi(s)\da$, and consider the execution of $\Psi(s+1)$.  The execution of $\Psi(s+1)$ reaches iteration $s+1$ of the main loop because $\Psi(s)\da$.  If iteration $s+1$ enters the `else' case, then clearly $\Psi(s+1)\da$.  If iteration $s+1$ enters the `if' case, then it executes $A^{\lc \eta \rc}_C(s+1)$.  However, $\la \lc \pi \rc, e \ra \in C$ because it entered $C$ during iteration $s_0 < s+1$.  Thus from the true $\Pi_1$ sentence $\true(\lc \pi \rc)$, the assumption $\tot(\Phi_e)$, and Lemma~\ref{lem-helperA} item~(i), we conclude that $A^{\lc \eta \rc}_C(s+1)\da$.  Thus $\Psi(s+1)\da$.  This completes the induction.  Finally, we conclude $\tot(\Psi)$ from $(\forall s \leq s_0)(\Psi(s)\da)$ and $(\forall s \geq s_0)(\Psi(s)\da)$.  Thus $T^+ \vdash \tot(\Phi_e) \imp \tot(\Psi)$, as desired.
\end{proof}

Let $\dg a = \degP{T^+}(\tot(\Psi))$.  Then $\dg a \meet \dg b \geP{T^+} \dg 0$ by Claim~\ref{claim-NotZeroT}.  If $\dg c$ is such that $\dg a \join \dg c \geqP{T^+} \dg b$, then $\dg c \geqP{T^+} \dg a$ by Claim~\ref{claim-avoidT}, so $\dg c \geqP{T^+} \dg b$.
\end{proof}

\begin{Corollary}\label{cor-NoCuppingT}
Let $T$ be a recursively axiomatizable extension of $\iso$.  For every $\dg b \in \Pd{T^+}$ with $\dg b \geP{T^+} \dg 0$, there is an $\dg a \in \Pd{T^+}$ with $\dg 0 \leP{T^+} \dg a \leP{T^+} \dg b$ such that $(\forall \dg c \in \Pd{T^+})[(\dg a \join \dg c = \dg b) \imp (\dg c = \dg b)]$.  That is, for every $\dg b \geP{T^+} \dg 0$, there is a non-zero $\dg a \leP{T^+} \dg b$ that does not cup to $\dg b$.
\end{Corollary}

\begin{proof}
Given $\dg b \geP{T^+} \dg 0$, by Theorem~\ref{thm-NoCupHelpT}, let $\dg x$ be such that $\dg x \meet \dg b \geP{T^+} \dg 0$ and $(\forall \dg c)[(\dg x \join \dg c \geqP{T^+} \dg b) \imp (\dg c \geqP{T^+} \dg b)]$.  Let $\dg a = \dg x \meet \dg b$.  Then $\dg 0 \leP{T^+} \dg a \leP{T^+} \dg b$.  Consider a $\dg c \in \Pd{T^+}$ such that $\dg a \join \dg c = \dg b$.  Clearly $\dg c \leqP{T^+} \dg b$.  On the other hand, using the fact that $\Pd{T^+}$ is a distributive lattice,
\begin{align*}
\dg b = \dg a \join \dg c = (\dg x \meet \dg b) \join \dg c = (\dg x \join \dg c) \meet (\dg b \join \dg c) = (\dg x \join \dg c) \meet \dg b.
\end{align*}
Thus $\dg x \join \dg c \geqP{T^+} \dg b$, which implies that $\dg c \geqP{T^+} \dg b$ by the choice of $\dg x$.  Thus $\dg c = \dg b$.
\end{proof}

Theorem~\ref{thm-NoCupHelpT} and Corollary~\ref{cor-NoCuppingT} also hold with $T$ in place of $T^+$.  In this situation, the definition of $\Psi$ can be simplified because there is no longer any need for the $\pi$'s and $\eta$'s.  One must be careful to check that a similar verification can be done using only $T$.

\section*{Acknowledgments}
We thank Mingzhong Cai, Lars Kristiansen, Robert Lubarsky, Jan-Christoph Schlage-Puchta, and Andreas Weiermann for their helpful comments on the drafts of this work.

\bibliographystyle{amsplain}
\bibliography{cupping}

\vfill

\end{document}